\documentclass[proceedings,submission%
]{dmtcs}

\usepackage[latin1]{inputenc}
\usepackage{subfigure}

\usepackage[numbers,round]{natbib}

\usepackage{amsmath,amssymb,theorem}

\newcommand{\assign}{:=}
\newcommand{\bignone}{}
\newcommand{\nin}{\not\in}
\newcommand{\tmop}[1]{\ensuremath{\operatorname{#1}}}
\newcommand{\tmtextit}[1]{{\itshape{#1}}}

\newtheorem{theorem}{Theorem}
\newtheorem{lemma}[theorem]{Lemma}

{\theorembodyfont{\rmfamily}
\newtheorem{remark}[theorem]{Remark}
\newtheorem{example}[theorem]{Example}}

\renewcommand{\ge}{\geqslant}
\renewcommand{\le}{\leqslant}

\newcommand{\qn}[2]{\ensuremath{\left[ {#1} \right]_{#2}}}
\newcommand{\qbinom}[3]{\ensuremath{\binom{#1}{#2}_{#3}}}
\newcommand{\modulo}[1]{\hspace{1em} \tmop{mod} {#1}}

\title{A $q$-analog of Ljunggren's binomial congruence}

\author{Armin Straub\addressmark{1}\thanks{Partially supported by grant NSF-DMS 0713836.}}

\address{\addressmark{3}Tulane University, New Orleans, USA. Email: \texttt{astraub@tulane.edu}}


\begin{document}
\maketitle

\begin{abstract}
\paragraph{Abstract.}
  We prove a $q$-analog of a classical binomial congruence due to Ljunggren
  which states that
  \[ \binom{a p}{b p} \equiv \binom{a}{b} \]
  modulo $p^3$ for primes $p\ge5$.  This congruence subsumes and builds on
  earlier congruences by Babbage, Wolstenholme and Glaisher for which we recall
  existing $q$-analogs. Our congruence generalizes an earlier result of Clark.

\paragraph{R\'esum\'e.}
  to be added


\end{abstract}

\section{Introduction and notation}

Recently, $q$-analogs of classical congruences have been studied by several
authors including {\cite{clark-qbin95}}, {\cite{andrews-qcong99}},
{\cite{shipan-qwolst07}}, {\cite{pan-qlehmer07}}, {\cite{chapman-qwilson08}},
{\cite{dilcher-qharm08}}. Here, we consider the classical congruence
\begin{equation}\label{eq:classical}
  \binom{a p}{b p} \equiv \binom{a}{b} \modulo{p^3}
\end{equation}
which holds true for primes $p \geqslant 5$. This also appears as Problem 1.6
(d) in {\cite{stanley-ec1}}. Congruence (\ref{eq:classical}) was proved in
1952 by Ljunggren, see {\cite{granville-bin97}}, and subsequently generalized
by Jacobsthal, see Remark \ref{rk-jacobsthal}.

Let $[n]_q \assign 1 + q + \ldots q^{n - 1}$, $[n]_q ! \assign [n]_q [n - 1]_q
\cdots [1]_q$ and
\[ \qbinom{n}{k}{q} \assign \frac{[n]_q !}{[k]_q ! [n - k]_q !} \]
denote the usual $q$-analogs of numbers, factorials and binomial coefficients
respectively. Observe that $\qn{n}{1}=n$ so that in the case $q=1$ we recover
the usual factorials and binomial coefficients as well. Also, recall that the
$q$-binomial coefficients are polynomials in $q$ with nonnegative integer
coefficients. An introduction to these $q$-analogs can be found in
\cite{stanley-ec1}.

We establish the following $q$-analog of \eqref{eq:classical}:
\begin{theorem}\label{thm:q}
  For primes $p\ge5$ and nonnegative integers $a, b$,
  \begin{equation}\label{eq:qclassical}
    \qbinom{ap}{bp}{q} \equiv \qbinom{a}{b}{q^{p^2}} - \binom{a}{b + 1} \binom{b + 1}{2}
    \frac{p^2 - 1}{12} (q^p - 1)^2 \modulo{\qn{p}{q}^3} .
  \end{equation}
\end{theorem}
The congruence \eqref{eq:qclassical} and similar ones to follow are to be
understood over the ring of polynomials in $q$ with integer coefficients. We remark
that $p^2-1$ is divisible by $12$ for all primes $p\ge5$.

Observe that \eqref{eq:qclassical} is indeed a $q$-analog of
\eqref{eq:classical}: as $q\to1$ we recover \eqref{eq:classical}.

\begin{example}
  Choosing $p=13$, $a=2$, and $b=1$, we have
  \begin{align*}
    \qbinom{26}{13}{q} &= 1+q^{169} - 14 (q^{13}-1)^2
    + (1+q+\ldots+q^{12})^3 f(q)
  \end{align*}
  where $f(q) = 14-41q+41q^2-\ldots+q^{132}$ is an irreducible polynomial with
  integer coefficients. Upon setting $q=1$, we obtain $\binom{26}{13}\equiv2$
  modulo $13^3$.
\end{example}

Since our treatment very much parallels the classical case, we give a brief
history of the congruence \eqref{eq:classical} in the next section before
turning to the proof of Theorem \ref{thm:q}.

\section{A bit of history}

A classical result of Wilson states that $(n-1)! + 1$ is divisible by $n$ if
and only if $n$ is a prime number. ``In attempting to discover some analogous
expression which should be divisible by $n^2$, whenever $n$ is a prime, but
not divisible if $n$ is a composite number'', \cite{babbage}, Babbage is led to
the congruence
\begin{equation}\label{eq:babbage}
  \binom{2p-1}{p-1} \equiv 1 \modulo{p^2}
\end{equation}
for primes $p\ge3$. In 1862 Wolstenholme, \cite{wolstenholme}, discovered
\eqref{eq:babbage} to hold modulo $p^3$, ``for several cases, in testing
numerically a result of certain investigations, and after some trouble
succeeded in proving it to hold universally'' for $p\ge5$. To this end, he
proves the fractional congruences
\begin{align}
  \sum_{i=1}^{p-1} \frac{1}{i} &\equiv 0 \modulo{p^2}, \label{eq:wol1}\\
  \sum_{i=1}^{p-1} \frac{1}{i^2} &\equiv 0 \modulo{p} \label{eq:wol2}
\end{align}
for primes $p\ge5$. Using \eqref{eq:wol1} and \eqref{eq:wol2} he then extends
Babbage's congruence \eqref{eq:babbage} to hold modulo $p^3$:
\begin{equation}\label{eq:wolstenholme}
  \binom{2p-1}{p-1} \equiv 1 \modulo{p^3}
\end{equation}
for all primes $p\ge5$. Note that \eqref{eq:wolstenholme} can be rewritten as
$\binom{2p}{p} \equiv 2$ modulo $p^3$.  The further generalization of
\eqref{eq:wolstenholme} to \eqref{eq:classical}, according to
\cite{granville-bin97}, was found by Ljunggren in 1952. The case $b=1$ of
\eqref{eq:classical} was obtained by Glaisher, \cite{glaisher}, in 1900.

In fact, Wolstenholme's congruence \eqref{eq:wolstenholme} is central to the
further generalization \eqref{eq:classical}. This is just as true when
considering the $q$-analogs of these congruences as we will see here in Lemma
\ref{lem:2}.

A $q$-analog of the congruence of Babbage has been found by Clark
{\cite{clark-qbin95}} who proved that
\begin{equation}\label{eq:clark}
  \qbinom{ap}{bp}{q} \equiv \qbinom{a}{b}{q^{p^2}} \modulo{\qn{p}{q}^2} .
\end{equation}
We generalize this congruence to obtain the $q$-analog \eqref{eq:qclassical} of
Ljunggren's congruence \eqref{eq:classical}.  A result similar to
\eqref{eq:clark} has also been given by Andrews in \cite{andrews-qcong99}.

Our proof of the $q$-analog proceeds very closely to the history just outlined.
Besides the $q$-analog \eqref{eq:clark} of Babbage's congruence
\eqref{eq:babbage} we will employ $q$-analogs of Wolstenholme's harmonic
congruences \eqref{eq:wol1} and \eqref{eq:wol2} which were recently supplied by
Shi and Pan, \cite{shipan-qwolst07}:
\begin{theorem}\label{thm:shipan}
  For primes $p\ge5$,
  \begin{equation}\label{eq:qwol1}
    \sum_{i = 1}^{p - 1} \frac{1}{[i]_q} \equiv - \frac{p - 1}{2} (q
    - 1) + \frac{p^2 - 1}{24} (q - 1)^2 [p]_q \modulo{\qn{p}{q}^2}
  \end{equation}
  as well as
  \begin{equation}\label{eq:qwol2}
     \sum_{i = 1}^{p - 1} \frac{1}{[i]_q^2} \equiv - \frac{(p - 1) (p
    - 5)}{12} (q - 1)^2 \modulo{\qn{p}{q}}.
  \end{equation}
\end{theorem}
This generalizes an earlier result {\cite{andrews-qcong99}} of Andrews.

\section{A $q$-analog of Ljunggren's congruence}

In the classical case, the typical proof of Ljunggren's congruence
\eqref{eq:classical} starts with the Chu-Vandermonde identity which has the
following well-known $q$-analog:

\begin{theorem}\label{thm:qchu}
  \begin{equation*}
    \qbinom{m+n}{k}{q} = \sum_j \qbinom{m}{j}{q} \qbinom{n}{k-j}{q} q^{j(n-k+j)}.
  \end{equation*}
\end{theorem}

We are now in a position to prove the $q$-analog of \eqref{eq:classical}.

\begin{proof}[of Theorem \ref{thm:q}]
  As in {\cite{clark-qbin95}} we start with the identity
  \begin{equation}
    \qbinom{ap}{bp}{q} = \sum_{c_1 + \ldots + c_a = b p}
    \qbinom{p}{c_1}{q} \qbinom{p}{c_2}{q} \cdots \qbinom{p}{c_a}{q}
    q^{^{p \sum_{1 \leqslant i \leqslant a} (i - 1) c_i
    - \sum_{1 \leqslant i < j \leqslant a} c_i c_j}}
    \label{eq-qchu}
  \end{equation}
  which follows inductively from the $q$-analog of the Chu-Vandermonde identity
  given in Theorem \ref{thm:qchu}.  The summands which are not divisible by
  $[p]_q^2$ correspond to the $c_i$ taking only the values $0$ and $p$. Since
  each such summand is determined by the indices $1 \le j_1 < j_2 < \ldots <
  j_b \le a$ for which $c_i = p$, the total contribution of these terms is
  \[ \sum_{1 \leqslant j_1 < \ldots < j_b \leqslant a} q^{p^2 \sum_{k = 1}^b
     (j_k - 1) - p^2 \binom{b}{2}} = \sum_{0 \leqslant i_1
     \leqslant \ldots \leqslant i_b \leqslant a - b} q^{p^2 \sum_{k = 1}^b i_k}
     = \qbinom{a}{b}{q^{p^2}} . \]
  This completes the proof of \eqref{eq:clark} given in {\cite{clark-qbin95}}.
  
  To obtain \eqref{eq:qclassical} we now consider those summands in
  (\ref{eq-qchu}) which are divisible by $[p]_q^2$ but not divisible by
  $[p]_q^3$. These correspond to all but two of the $c_i$ taking values $0$ or
  $p$. More precisely, such a summand is determined by indices $1 \leqslant
  j_1 < j_2 < \ldots < j_b < j_{b + 1} \leqslant a$, two subindices $1
  \leqslant k < \ell \leqslant b + 1$, and $1 \leqslant d \leqslant p - 1$
  such that
  \[ c_i = \left\{ \begin{array}{l}
       d \text{ for $i = j_k$},\\
       p - d \text{ for $i = j_{\ell}$},\\
       p \text{ for $i \in \{j_1, \ldots, j_{b + 1} \}\backslash\{j_k,
       j_{\ell} \}$},\\
       0 \text{ for $i \nin \{j_1, \ldots, j_{b + 1} \}$} .
     \end{array} \right. \]
  For each fixed choice of the $j_i$ and $k, \ell$ the contribution of the
  corresponding summands is
  \[ \sum_{d = 1}^{p - 1} \qbinom{p}{d}{q} \qbinom{p}{p-d}{q}
      q^{p \sum_{1 \leqslant i \leqslant a} (i - 1) c_i
      - \sum_{1 \leqslant i < j \leqslant a} c_i c_j} \]
  which, using that $q^p \equiv 1$ modulo $\qn{p}{q}$, reduces modulo $\qn{p}{q}^3$ to
  \[ \sum_{d = 1}^{p - 1} \qbinom{p}{d}{q} \qbinom{p}{p-d}{q} q^{d^2}
      = \qbinom{2p}{p}{q} - [2]_{q^{p^2}} . \]
  We conclude that
  \begin{equation}\label{eq:cong2}
    \qbinom{ap}{bp}{q} \equiv \qbinom{a}{b}{q^{p^2}} + \binom{a}{b + 1}
      \binom{b + 1}{2} \left( \qbinom{2p}{p}{q} - [2]_{q^{p^2}} \right)
      \modulo{\qn{p}{q}^3}.
  \end{equation}
  The general result therefore follows from the special case $a = 2$, $b = 1$
  which is separately proved next.
\end{proof}

\section{A $q$-analog of Wolstenholme's congruence}

We have thus shown that, as in the classical case, the congruence
\eqref{eq:qclassical} can be reduced, via \eqref{eq:cong2}, to the case $a=2$,
$b=1$. The next result therefore is a $q$-analog of Wolstenholme's congruence
\eqref{eq:wolstenholme}.

\begin{lemma}\label{lem:2}
  For primes $p\ge5$,
  \[ \qbinom{2p}{p}{q} \equiv [2]_{q^{p^2}} - \frac{p^2 - 1}{12} (q^p -
    1)^2 \modulo{\qn{p}{q}^3} . \]
\end{lemma}

\begin{proof}
  Using that $\qn{a n}{q} = \qn{a}{q^n} \qn{n}{q}$ and $\qn{n + m}{q} =
  \qn{n}{q} + q^n \qn{m}{q}$ we compute
  \[ \qbinom{2p}{p}{q} = \frac{\qn{2 p}{q} \qn{2 p - 1}{q} \cdots \qn{p + 1}{q}}
     {\qn{p}{q} \qn{p - 1}{q} \cdots \qn{1}{q}} = \frac{\qn{2}{q^p}}{\qn{p - 1}{q} !} \prod_{k = 1}^{p
     - 1} \left( \qn{p}{q} + q^p \qn{p - k}{q} \right) \]
  which modulo $\qn{p}{q}^3$ reduces to (note that $\qn{p-1}{q} !$ is relatively
  prime to $\qn{p}{q}^3$)
  \begin{equation}
    \qn{2}{q^p}  \left( q^{(p - 1) p} + q^{(p - 2) p} \sum_{1 \leqslant i
    \leqslant p - 1} \frac{\qn{p}{q}}{\qn{i}{q}} + q^{(p - 3) p} \sum_{1
    \leqslant i < j \leqslant p - 1} \frac{\qn{p}{q} \qn{p}{q}}{\qn{i}{q} \qn{j}{q}}
    \right) . \label{eq-2ppharmonic}
  \end{equation}
  Combining the results \eqref{eq:qwol1} and \eqref{eq:qwol2} of Shi and Pan,
  \cite{shipan-qwolst07}, given in Theorem \ref{thm:shipan}, we deduce that for
  primes $p\ge5$,
  \begin{equation}
    \sum_{1 \leqslant i < j \leqslant p - 1} \frac{1}{\qn{i}{q} \qn{j}{q}} \bignone
    \equiv \frac{(p - 1) (p - 2)}{6} (q - 1)^2 \modulo{\qn{p}{q}} .
  \end{equation}
  Together with \eqref{eq:qwol1} this allows us to rewrite
  (\ref{eq-2ppharmonic}) modulo $\qn{p}{q}^3$ as
  \begin{align*}
    && \qn{2}{q^p}  \left( q^{(p - 1) p} + q^{(p - 2) p} \left( - \frac{p -
    1}{2} (q^p - 1) + \frac{p^2 - 1}{24} (q^p - 1)^2 \right) + \right.\\
    && \left. + q^{(p - 3) p} \frac{(p - 1) (p - 2)}{6} (q^p - 1)^2 \right).
  \end{align*}
  Using the binomial expansion
  \[ q^{m p} = ((q^p - 1) + 1)^m = \sum_k \binom{m}{k} (q^p - 1)^k \]
  to reduce the terms $q^{m p}$ as well as $\qn{2}{q^p} = 1 + q^p$ modulo the
  appropriate power of $\qn{p}{q}$ we obtain
  \[ \qbinom{2p}{p}{q} \equiv 2 + p (q^p - 1) + \frac{(p - 1) (5 p -
     1)}{12} (q^p - 1)^2 \modulo{\qn{p}{q}^3} . \]
  Since
  \[ [2]_{q^{p^2}} \equiv 2 + p (q^p - 1) + \frac{(p - 1) p}{2} (q^p - 1)^2
     \modulo{\qn{p}{q}^3} \]
  the result follows.
\end{proof}

\begin{remark}
  \label{rk-jacobsthal}Jacobsthal, see {\cite{granville-bin97}}, generalized
  the congruence (\ref{eq:classical}) to hold modulo $p^{3 + r}$ where $r$ is
  the $p$-adic valuation of
  \[ a b (a - b) \binom{a}{b} = 2 a \binom{a}{b + 1} \binom{b + 1}{2} . \]
  It would be interesting to see if this generalization has a nice analog in
  the $q$-world.
\end{remark}

\acknowledgements
Most parts of this paper have been written during a visit of the author at
Grinnell College. The author wishes to thank Marc Chamberland for his
encouraging and helpful support. Partial support of grant NSF-DMS 0713836 is
also thankfully acknowledged.

\end{document}